\documentclass[10pt]{amsart}

\usepackage{amssymb}
\usepackage{amsthm}
\usepackage{amsmath}

\usepackage[foot]{amsaddr}

\usepackage{hyperref}

\theoremstyle{plain}
\newtheorem{proposition}{Proposition}[section]
\newtheorem{theorem}[proposition]{Theorem}
\newtheorem{lemma}[proposition]{Lemma}
\newtheorem{corollary}[proposition]{Corollary}
\theoremstyle{definition}
\newtheorem{example}[proposition]{Example}
\newtheorem{definition}[proposition]{Definition}

\theoremstyle{remark}
\newtheorem{remark}[proposition]{Remark}
\newtheorem{conjecture}[proposition]{Conjecture}
\newtheorem*{problem}{Problem}

\DeclareMathOperator{\Aut}{Aut}

\DeclareMathOperator{\SL}{SL}
\DeclareMathOperator{\GL}{GL}
\DeclareMathOperator{\Gr}{Gr}
\DeclareMathOperator{\SO}{SO}

\DeclareMathOperator{\UU}{U}
\DeclareMathOperator{\PU}{PU}
\DeclareMathOperator{\PP}{P}

\DeclareMathOperator{\PGL}{PGL}

\DeclareMathOperator{\End}{End}

\DeclareMathOperator{\Spanset}{Span} 
\DeclareMathOperator{\Proj}{Proj} 
\DeclareMathOperator{\Id}{Id} 
\DeclareMathOperator{\rank}{rank} 
 
\DeclareMathOperator{\Lin}{Lin} 
 
\DeclareMathOperator{\hol}{hol}
\DeclareMathOperator{\dev}{dev} 
 
\DeclareMathOperator{\Isom}{Isom} 
\DeclareMathOperator{\id}{id}

\DeclareMathOperator{\Ac}{\mathcal{A}}
\DeclareMathOperator{\Bc}{\mathbb{B}}
\DeclareMathOperator{\Cc}{\mathcal{C}}

\DeclareMathOperator{\Oc}{\mathcal{O}}

\DeclareMathOperator{\Ab}{\mathbb{A}}
\DeclareMathOperator{\Cb}{\mathbb{C}}

\DeclareMathOperator{\Hb}{\mathbb{H}}

\DeclareMathOperator{\Kb}{\mathbb{K}}

\DeclareMathOperator{\Nb}{\mathbb{N}}
\DeclareMathOperator{\Pb}{\mathbb{P}}

\DeclareMathOperator{\Rb}{\mathbb{R}}

\DeclareMathOperator{\gL}{\mathfrak{g}}

\newcommand{\abs}[1]{\left|#1\right|}

\newcommand{\norm}[1]{\left\|#1\right\|}
\newcommand{\wt}[1]{\widetilde{#1}}
\newcommand{\wh}[1]{\widehat{#1}}

\begin{document}

\title[Proper quasi-homogeneous domains]{Proper quasi-homogeneous domains in flag manifolds and geometric structures}
\author{Andrew Zimmer}\address{Department of Mathematics, College of William and Mary, Williamsburg, VA 23185.}
\email{amzimmer@wm.edu}
\date{\today}
\subjclass[2010]{}

\begin{abstract} 
In this paper we study domains in flag manifolds which are bounded in an affine chart and whose projective automorphism group acts co-compactly. In contrast to the many examples in real projective space, we will show that no examples exist in many flag manifolds. Moreover, in the cases where such domains can exist, we show that they satisfy a natural convexity condition and have an invariant metric which generalizes the Hilbert metric. As an application we give some restrictions on the developing map for certain $(G,X)$-structures. 
 \end{abstract}

\maketitle

\section{Introduction}

Suppose $G$ is a connected semisimple Lie group with trivial center and without compact factors. If $P \leq G$ is a parabolic subgroup, then $G$ acts by 
diffeomorphisms on the compact manifold $G/P$. Given an open set $\Omega \subset G/P$ we define the \emph{automorphism group of $\Omega$} to be 
\begin{align*}
\Aut(\Omega) = \{ g \in G : g \Omega =\Omega\}.
\end{align*}
An open set $\Omega \subset G/P$ is called \emph{quasi-homogeneous} if $\Aut(\Omega) \cdot K = \Omega$ for some compact subset $K \subset \Omega$. This paper is concerned with the geometry and classification of quasi-homogeneous domains.

\subsection{Uniformization} One motivation for studying quasi-homogeneous domains comes from the well known uniformization of Riemann surfaces: any Riemann surface $\Sigma$ can be identified with a quotient $\Gamma \backslash \Omega$ where $\Omega$ is a domain in the complex projective plane $\Pb(\Cb^2)$ and $\Gamma \leq \PGL_{2}(\Cb)$ is a discrete group which acts freely and properly discontinuously on $\Omega$. 

It seems natural to ask what happens in higher dimensions. In particular given a closed manifold $M$, can we identify $M$ with a  quotient $\Gamma \backslash \Omega$ where $\Omega$ is a domain in some flag manifold $G/P$ and $\Gamma \leq \Aut(\Omega)$? Since $M$ is compact, in this case the domain $\Omega$ will be quasi-homogeneous. 

\subsection{$(G,X)$-structures} Another (related) motivation for studying quasi-homogeneous domains comes from the theory of geometric structures on compact manifolds.

Suppose $G$ is a Lie group acting transitively on a manifold $X$.  A \emph{$(G,X)$-structure} on a manifold $M$ is an open cover $M = \cup_{\alpha} U_\alpha$ along with coordinate charts $\varphi_\alpha : U_{\alpha} \rightarrow X$ such that the transition functions $\varphi_{\alpha} \circ \varphi_{\beta}^{-1}$ coincide locally with the restriction of an element in $G$ on $\varphi_{\beta}(U_\beta \cap U_{\alpha})$. 

Given a $(G,X)$-structure on a manifold $M$, one can ``unfold'' the structure to obtain a local diffeomorphism $\dev : \wt{M} \rightarrow X$ from the universal cover $\wt{M}$ of $M$ to $X$ called the \emph{developing map} and a homomorphism $\hol: \pi_1(M,m) \rightarrow G$ called the \emph{holonomy map}. The map $\dev$ will be $\hol$-equivariant and when $\dev$ is a diffeomorphism onto its image we can identify $M$ with $\Gamma \backslash \Omega$ where $\Gamma = \hol(\pi_1(M,m))$ and $\Omega = \dev(\wt{M})$. See for instance~\cite{G1988} for more details.

When $M$ is closed, the group $\hol(\pi_1(M,m))$ acts co-compactly on $\dev(\wt{M})$ and thus $\dev(\wt{M})$ is a quasi-homogeneous domain in $X$. Our results about quasi-homogeneous domains will imply the following.

\begin{theorem}\label{thm:geom_str_main}
 Suppose $M$ is a closed manifold, $G$ is a connected non-compact simple Lie group with trivial center, and $P \leq G$ is a non-maximal parabolic subgroup. If $\{ (U_\alpha, \varphi_\alpha)\}_{\alpha \in \Ac}$ is a $(G,G/P)$-structure on $M$, then the image of the developing map cannot be bounded in an affine chart.
\end{theorem}

Theorem~\ref{thm:geom_str_main} is an immediate consequence of Theorem~\ref{thm:proper_non_maximal} below. 

\begin{theorem}\label{thm:covering_map}
 Suppose $M$ is a closed manifold, $G$ is a connected semi-simple Lie group with trivial center and no compact factors, and $P \leq G$ is a parabolic subgroup. If $\{ (U_\alpha, \varphi_\alpha)\}_{\alpha \in \Ac}$ is a $(G,G/P)$-structure on $M$ and the image of the developing map is bounded in an affine chart, then $\dev : \wt{M} \rightarrow \dev(\wt{M})$ is a covering map.
\end{theorem}

In Proposition~\ref{prop:proper_ss} below, we will show that when $\Omega$ is bounded in an affine chart the group $\Aut(\Omega)$ acts properly on $\Omega$. Using this fact it is straightforward to establish  Theorem~\ref{thm:covering_map}, see Section~\ref{sec:covering_map} for details. 

\subsection{Proper domains}

In this paper we will restrict our attention to a particular class of domains in flag manifolds: 

\begin{definition}  An open set $\Omega \subset G/P$ is called a \emph{proper domain} if it is connected and bounded in an affine chart of $G/P$. 
\end{definition}

There are (at least) three reasons for restricting our attention to these domains: 

\begin{enumerate}
\item As mentioned above, every Riemann surface can be identified with a quotient $\Gamma \backslash \Omega$ where $\Omega \subset \Pb(\Cb^2)$ and $\Gamma \leq \Aut(\Omega)$ is a discrete group. In fact, we can always assume that $\Omega$ is either $\Pb(\Cb^2)$, an affine chart in $\Pb(\Cb^2)$, or the unit disk in an affine chart of $\Pb(\Cb^2)$. Each of these three cases lead to very different classes of surfaces, so when seeking higher dimensional analogues of uniformization it makes sense to try and specialize to one of the three cases.
\item  There is a rich theory of \emph{convex divisible domains} in real projective space. A proper convex set $\Omega \subset \Pb(\Rb^{d+1})$ is called \emph{divisible} when there exists a discrete group $\Gamma \leq \Aut(\Omega)$ which acts co-compactly, freely, and properly on $\Omega$. The symmetric domain $\Bc_{d,1}$ defined in Example~\ref{ex:symmetric} below is the fundamental example of a convex divisible domain but there are many non-homogeneous examples, see the survey papers by Benoist~\cite{B2008}, Marquis~\cite{M2013}, and Quint~\cite{Q2010}. It seems very natural to attempt to extend this theory to other flag manifolds and a key feature of these domains is the fact that they are bounded in affine charts. 
\item The fact that we restrict our attention to domains which are bounded in affine charts also allows us to adapt some techniques used in several complex variables to study the bi-holomorphism group of bounded domains in complex Euclidean space (see the survey paper~\cite{IK1999}). 
\end{enumerate}

 As the next example shows certain symmetric spaces give rise to homogeneous proper domains in certain Grassmanians: 

\begin{example}\label{ex:symmetric} Let $\Kb$ be either the real numbers $\Rb$, the complex numbers $\Cb$, or the quaternions $\Hb$. If $G = \PGL_d(\Kb)$ and $P \leq \PGL_d(\Kb)$ is a parabolic subgroup, then $P$ is the stabilizer of some $\Kb$-flag and $G/P$ can be identified with a 
 flag manifold. In the particular case when $P$ is a maximal parabolic subgroup of $G$, then $P$ is the stabilizer some $p$-plane in $\Kb^{d}$. Let $q=d-p$. Then we can identify $G/P$ with $\Gr_p(\Kb^{p+q})$ the Grassmanian of $p$-planes in $\Kb^{p+q}$. 
Let $M_{p+q,p}(\Kb)$ be the space of $(p+q)$-by-$p$ matrices with entries in $\Kb$. We can then identify the quotient manifold
\begin{align*}
\left\{ A \in M_{(p+q), p}(\Kb) : \rank A = p \right\} / \GL_p(\Kb)
\end{align*}
with $\Gr_p(\Kb^{p+q})$ via $A \rightarrow \operatorname{Im}(A)$. Then 
\begin{align*}
\Ab := \left\{ \begin{bmatrix} \Id_p \\ X \end{bmatrix} : X \in M_{q,p}(\Kb)\right\}  \subset \Gr_{p}(\Kb^{p+q})
\end{align*}
is an affine chart of $\Gr_p(\Kb^{p+q})$ and so the open set
\begin{align*}
\Bc_{p,q} := \left\{ \begin{bmatrix} \Id_p \\ X \end{bmatrix} : \norm{X} < 1\right\} \subset \Gr_{p}(\Kb^{p+q})
\end{align*}
is a proper domain. Next let $\UU_{\Kb}(p,q) \leq  \GL_{d}(\Kb)$ be the group which preserves the form
\begin{align*}
\overline{x}_1y_1+\dots + \overline{x}_py_p - \overline{x}_{p+1} y_{p+1} - \dots - \overline{x}_{p+q} y_{p+q}.
\end{align*}
Then $\Aut(\Bc_{p,q})$ coincides with $\PU_{\Kb}(p,q)$, the image of $\UU_{\Kb}(p,q)$ in $\PGL_d(\Kb)$. Moreover $\Aut(\Bc_{p,q})$ acts 
acts transitively on $\Bc_{p,q}$ and the stabilizer of $\begin{bmatrix} \Id_p & 0 \end{bmatrix}^t$ 
is the group $\PP( \UU_{\Kb}(p) \times \UU_{\Kb}(q))$, so we can identify
\begin{align*}
\Bc_{p,q} \cong \PU_{\Kb}(p,q) / \PP( \UU_{\Kb}(p) \times \UU_{\Kb}(q)).
\end{align*}
In particular, $\Bc_{p,q}$ is a geometric model of the symmetric space associated to $ \PU_{\Kb}(p,q)$. When $\Kb=\Rb$ and $q=1$, this is the 
Klein-Beltrami model of real hyperbolic $p$-space. 
\end{example}

\subsection{Non-maximal parabolic subgroups}

In contrast to the above examples, our main rigidity result shows that many flag manifolds have no quasi-homogenous proper domains:

\begin{theorem}\label{thm:proper_non_maximal}(see Section~\ref{sec:main_thm} below)
Suppose $G$ is a connected non-compact simple Lie group with trivial center and $P \leq G$ is a non-maximal parabolic subgroup. If $\Omega \subset G/P$ is a proper domain, then $\Aut(\Omega)$ cannot act co-compactly on $\Omega$. 
\end{theorem}

\begin{remark} \
\begin{enumerate}
\item The theorem fails for general semi-simple groups. Take for instance $G=\PGL_{p+q}(\Rb) \times \PGL_{p+q}(\Rb)$ and 
let $P \leq \PGL_{p+q}(\Rb)$ be the stabilizer of some $p$-plane. Then we can identify $(G \times G)/(P \times P)$ with $\Gr_p(\Rb^{p+q}) \times \Gr_p(\Rb^{p+q})$ and 
\begin{align*}
\Bc_{p,q} \times \Bc_{p,q} \subset \Gr_p(\Rb^{p+q}) \times \Gr_p(\Rb^{p+q})
\end{align*}
is a proper quasi-homogeneous domain. Notice that the parabolic subgroup $P \times P$ is non-maximal: it is contained in 
$\PGL_{p+q}(\Rb) \times P$ and $P \times \PGL_{p+q}(\Rb)$. The semisimple case will be explored in more detail in Theorem~\ref{thm:products} below.
\item There are examples of non-proper quasi-homogeneous domains: suppose $G = \PGL_{p+q}(\Rb)$ and $P \leq G$ is the stabilizer of a $p$-plane in $\Kb^{p+q}$. Then, by Example~\ref{ex:symmetric}, there exists a proper homogeneous domain $\Omega \subset G/P$. Now consider a parabolic subgroup $P^\prime \lneqq  P$ and the natural projection map $\pi: G/P^\prime \rightarrow G/P$.  Then $\pi^{-1}(\Omega)$ is a homogeneous domain in $G/P^\prime$ which is not proper (see Proposition~\ref{prop:bigger_para} below). 
\item We should also mention recent constructions~\cite{GW2008, GW2012, GGKW2015, KLP2013, KLP2014, KLP2014b} of open domains $\Omega$ in 
certain flag manifolds where there exists a discrete group $\Gamma \leq \Aut(\Omega)$ which acts properly, freely, and cocompactly on $\Omega$. 
These constructions use the theory of Anosov representations and (to the best of our knowledge) never produce proper domains.
\end{enumerate}
\end{remark}

\subsection{The general semisimple case} Suppose $G$ is a connected semisimple Lie group with trivial center and no compact factors. Then there exists $G_1, \dots, G_r$ non-compact simple Lie groups each with trivial centers such that
\begin{align*}
G \cong \prod_{i=1}^r G_i.
\end{align*}
Now if $P \leq G$ is a parabolic subgroup we can find subgroups $P_i \leq G_i$ such that $P \cong \prod_{i=1}^r P_i$. Moreover, either $P_i = G_i$ or $P_i \leq G_i$ is a parabolic subgroup. Since we are interested in the flag manifold $G/P$, we further assume that $P_i \neq G_i$ for all $1 \leq i \leq r$.

\begin{theorem}\label{thm:products}(see Section~\ref{sec:products} below)
With the notation above, suppose $\Omega \subset G/P$ is a proper quasi-homogeneous domain. Then for $ 1 \leq i \leq r$ there exists a proper quasi-homogeneous domain $\Omega_i \subset G_i/P_i$ 
such that 
\begin{align*}
\Omega = \prod_{i=1}^r \Omega_i.
\end{align*}
In particular, each $P_i$ is a maximal parabolic subgroup of $G_i$.
\end{theorem}

This theorem reduces the study of quasi-homogeneous domains to the case when $G$ is simple. 

\subsection{The geometry of quasi-homogeneous domains}

Kobayashi proved the following theorem connecting symmetry with convexity for domains in real projective space.

\begin{theorem}\cite{K1977}\label{thm:kob_convex}
If $\Omega \subset \Pb(\Rb^{d+1})$ is a quasi-homogeneous proper domain, then $\Omega$ is convex. 
\end{theorem}

In real projective space, one usually defines convexity using projective lines: a domain is convex if its intersection with every projective line is connected. However there is a dual definition: a connected open set $\Omega$ is convex if and only if for every $x \in \partial \Omega$ there exists a hyperplane $H$ such that $x \in H$ and $H \cap \Omega = \emptyset$.

This dual definition generalizes in a natural way to flag manifolds. Suppose $G$ is a connected non-compact simple Lie group with trivial center
 and $P \leq G$ is a parabolic subgroup. Fix a parabolic subgroup $Q$ opposite to $P$. Then for $\xi = hQ \in G/Q$ define the subset $Z_\xi \subset G/P$ by
\begin{align*}
Z_\xi := \{ gP \in G/P:  gPg^{-1} \text{ is not transverse to } hQh^{-1}\}.
\end{align*}

\begin{remark} \ \begin{enumerate}
\item This definition does not depend on the choice of $Q$: if $Q_1, Q_2 \leq G$ are parabolic subgroups both opposite to a parabolic subgroup $P \leq G$, then $Q_1$ and $Q_2$ are conjugate. In particular, for any $\xi_1 \in G/Q_1$ there exists $\xi_2 \in G/Q_2$ such that $Z_{\xi_1} = Z_{\xi_2}$. 

\item Notice that $G/P - Z_\xi$ is an affine chart of $G/P$, see  Subsection~\ref{subsec:affine_chart}.
\end{enumerate} 
\end{remark}
\begin{example}\label{ex:dual_convex_proj} Again let $\Kb$ be either the real numbers $\Rb$, the complex numbers $\Cb$, or the quaternions $\Hb$. Let $e_1, \dots, e_d$ be the standard basis of $\Kb^d$ and let $P \leq \PGL_d(\Kb)$ be the stabilizer of the line $\Kb e_1$. Then $P \leq \PGL_d(\Rb)$ is parabolic and $G/P$ can be identified with $\Pb(\Kb^d)$. If $Q \leq \PGL_d(\Rb)$ is the stabilizer of $\Spanset_{\Kb}(e_2, \dots, e_d)$, then $Q$ is a parabolic subgroup opposite to $P$. Moreover we can identify $G/Q$ with $\Gr_{d-1}(\Kb^d)$. With these identifications, if $\xi \in \Gr_{d-1}(\Kb^d)$ then $Z_\xi$ is the image of the hypersurface $\xi$ in $\Pb(\Kb^d)$. 
\end{example}

Motivated by the hyperplane definition of convexity in real projective space we make the following definition.

\begin{definition}\label{defn:dual_convex}
Suppose $G$ is a connected non-compact simple Lie group with trivial center and $P \leq G$ is a parabolic subgroup. 
An open connected set $\Omega \subset G/P$ is called \emph{dual convex} if for each $x \in \partial \Omega$ there exists a parabolic subgroup $Q$ opposite to $P$ and some $\xi \in G/Q$ such that $x \in Z_\xi$ and $Z_\xi \cap \Omega = \emptyset$. 
\end{definition}

We will then prove the following generalization of Kobayashi's theorem.

\begin{theorem}\label{thm:dual_convex}(see Corollary~\ref{cor:dual_convex} below)
Suppose $G$ is a connected non-compact simple Lie group with trivial center  and $P \leq G$ is a parabolic subgroup. If $\Omega \subset G/P$ is a quasi-homogeneous proper domain, then $\Omega$ is dual convex. 
\end{theorem}

\begin{remark} Example~\ref{ex:dual_convex_proj} shows that for open connected sets in real projective space, dual convexity is equivalent to the standard definition of convexity. Dual convexity also generalizes a notion of convexity from several complex variables. In particular, an open set $\Omega \subset \Cb^d$ is often called \emph{weakly linearly convex} if for each point $x \in \partial \Omega$ there exists a complex hyperplane $H$ such that $x \in H$ and $H \cap \Omega = \emptyset$. Surprisingly, this weak form of convexity has strong analytic implications. See~\cite{APS2004, H2007} for more details. 
\end{remark}

\subsection{An invariant metric} 

Every proper convex set $\Omega \subset \Pb(\Rb^{d+1})$  has a metric $H_{\Omega}$ called the \emph{Hilbert metric} which is complete, geodesic, and $\Aut(\Omega)$-invariant. This metric is a useful tool understanding the geometry of domains with large projective symmetry groups. 

We will show that a proper dual convex domain in a flag manifold always has an complete $\Aut(\Omega)$-invariant metric which is a natural analogue of the Hilbert metric. 

\begin{theorem}\label{thm:intro_metric}Suppose $G$ is a connected non-compact simple Lie group with trivial center and $P \leq G$ is a parabolic subgroup. If $\Omega \subset G/P$ is a proper dual convex domain, then there exists an explicit  $\Aut(\Omega)$-invariant complete metric $C_\Omega$ which generates the standard topology on $\Omega$.
\end{theorem}

\begin{remark} \ \begin{enumerate}
\item The metric $C_\Omega$ can also be seen as a natural analogue of the Carath{\'e}odory metric from several complex variables.
\item For ``linearly convex'' domains in complex projective space, $C_\Omega$ was introduced by Dubois~\cite{D2008} and used in~\cite{Z2015b} to provide several characterizations of the unit ball. 
\end{enumerate} 
\end{remark}

 \subsection*{Acknowledgments} 

I would like to thank Pierre Clare, Lizhen Ji, Wouter Van Limbeek, and Ralf Spatzier for many interesting discussions. I would also like to thank the referee for a number of helpful comments and corrections. This material is based upon work supported by the National Science Foundation under grants DMS-1400919 and DMS-1760233.

\section{Examples and rigidity of proper quasi-homogeneous domains}

In this section we describe some examples of proper quasi-homogeneous domains.

\subsection{The symmetric case}\label{subsec:symmetric} The Borel embedding shows that every non-compact Hermitian symmetric space $X$ embeds as a domain $\Omega_X$ into a flag manifold $G/P$ (and this flag manifold can be identified with the compact dual of $X$) such that $\Aut(\Omega_X) = \Isom_0(X)$. The image of this embedding is a proper domain.  

More generally, Nagano~\cite[Theorem 6.1]{N1965} has characterized all the non-compact symmetric spaces $X$ whose compact dual $X^*$ can be identified with a flag manifold $G/P$ and $X$ embeds as a domain $\Omega_X$ into $G/P$ such that $\Aut(\Omega_X) = \Isom_0(X)$. In all these examples the images are proper domains~\cite[Theorem 6.2]{N1965}

There also exists examples of symmetric spaces which embed into real projective space as a proper domain. In particular, the symmetric spaces associated to $\SL_d(\Rb)$, $\SL_d(\Cb)$, $\SL_d(\Hb)$, and $E_{6(-26)}$ can all be realized as a proper homogeneous domains in a real projective space.  For instance, consider the convex set
\begin{align*}
\mathcal{P} = \{ [X] \in \Pb(S_{d,d}) : X \text{ is positive definite}\}
\end{align*}
where $S_{d,d}$ is the vector space of real symmetric $d$-by-$d$ matrices. Then the group $\SL_{d}(\Rb)$ acts transitively on $\mathcal{P}$ by $g \cdot [X] = [g X g^t]$ and the stabilizer of a point is $\SO(d)$. So we have an identification 
\begin{align*}
\mathcal{P} = \SL_d(\Rb) / \SO(d).
\end{align*}

There does not appear to be a general classification of embeddings of symmetric spaces into flag manifolds. 

\begin{problem} Characterize the symmetric spaces $X$ which embed as a proper domain $\Omega_X$ into a flag manifold where $\Isom_0(X) = \Aut(\Omega_X)$. \end{problem}

\subsection{Real projective space} Beyond the examples mentioned in the subsection above, there are a rich class of proper domains $\Omega \subset \Pb(\Rb^d)$ where $\Aut(\Omega)$ contains a discrete group $\Gamma$ which acts cocompactly on $\Omega$. By a result of Kobayashi (Theorem~\ref{thm:kob_convex} above) these examples will alway be convex and are often called \emph{convex divisible domains}. Here are some examples:
\begin{enumerate}
\item Let $\Bc\subseteq \Pb(\Rb^{d+1})$ be the Klein-Beltrami model of $\Hb^d_{\Rb}$. Results of Johnson-Millson~\cite{JM1987} and Koszul~\cite{K1968} imply that the domain $\Bc$ can be deformed to a divisible convex domain $\Omega$ where $\Aut(\Omega)$ is discrete (see~\cite[Section 1.3]{B2000} for $d>2$ and~\cite{G1990} for $d=2$).
\item For every $d \geq 4$, Kapovich~\cite{K2007} has constructed divisible convex domains $\Omega \subset \Pb(\Rb^{d+1})$ such that $\Aut(\Omega)$ is discrete, Gromov hyperbolic, and not quasi-isometric to any symmetric space,.
\item Benoist~\cite{B2006} and Ballas, Danciger, and Lee~\cite{BDL2015} have constructed divisible convex domains $\Omega \subset \Pb(\Rb^{4})$ such that $\Aut(\Omega)$ is discrete, not Gromov hyperbolic, and not quasi-isometric to any symmetric space. 
\end{enumerate}
More background can be found in the survey papers by Benoist~\cite{B2008}, Marquis~\cite{M2013}, and Quint~\cite{Q2010}.

\subsection{Complex projective space}

As Example~\ref{ex:symmetric} shows, there exists a proper homogeneous domain $\Bc \subset \Pb(\Cb^d)$ (which is a model of complex hyperbolic space, see for instance~\cite[Chapter 19]{M1973}). 

In $\Pb(\Cb^2)$ there do exist non-homogeneous proper domains which admit a co-compact action by a discrete group in $\Aut(\Omega)$. However if $\partial \Omega$ has weak regularity then a result of Bowen implies that $\Omega$ must be a symmetric domain:

\begin{theorem}\cite{B1979}
Suppose $\Omega \subset \Pb(\Cb^2)$ is a proper domain and $\partial \Omega$ is a Jordan curve with Hausdorff dimension one. If there exists a discrete group $\Gamma \leq \Aut(\Omega)$ which acts co-compactly on $\Omega$, then $\Omega$ is projectively isomorphic to $\Bc$. 
\end{theorem}

In $\Pb(\Cb^3)$ the co-compact case is even more rigid and recent work of Cano and Seade implies the following:

\begin{theorem}\cite{CS2014}
Suppose $\Omega \subset \Pb(\Cb^3)$ is a proper domain and $\Gamma \leq \Aut(\Omega)$ is a discrete group which acts co-compactly on $\Omega$. Then $\Omega$ is projectively isomorphic to $\Bc$. 
\end{theorem}

It is worth noting that Cano and Seade's proof relies on Kobayashi and Ochiai's~\cite{KO1980} classification of compact complex surfaces with a projective structure. In particular, it is unclear if Cano and Seade's result should extend to higher dimensions. However, it is known that every proper quasi-homogeneous domain with $C^1$ boundary is symmetric.

\begin{theorem}\cite{Z2015b}
Suppose $\Omega \subset \Pb(\Cb^{d+1})$ is a proper quasi-homogeneous domain. If $\partial \Omega$ is $C^{1}$, then $\Omega$ is projectively isomorphic to $\Bc$.
\end{theorem}

\subsection{Rigidity and convexity} The embeddings of symmetric spaces mentioned in Subsection~\ref{subsec:symmetric} are always convex in some affine chart, see for instance~\cite[Theorem 6.2]{N1965}. Thus it seems natural to consider proper quasi-homogeneous domains which are convex in some affine chart. In particular, we say a domain $\Omega \subset G/P$ is a \emph{convex divisible domain} if $\Omega$ is a bounded open convex set of some affine chart and there exists a discrete group $\Gamma \leq \Aut(\Omega)$ which acts properly discontinuously, freely, and co-compactly on $\Omega$. 

For some flag manifolds there are no non-homogeneous convex divisible domains. In particular, Frankel proved the following:

\begin{theorem}\cite{F1989}
Suppose $\Omega \subset \Cb^d$ is a bounded convex open set and there exists a discrete group $\Gamma$ of bi-holomorphic maps of $\Omega$ which acts properly discontinuously, freely, and co-compactly on $\Omega$. Then $\Omega$ is a bounded symmetric domain. 
\end{theorem}

Thus if $G/P$ has a complex structure such that $G$ acts on $G/P$ homomorphically we see that the only convex divisible domains in $G/P$ are homogeneous. Frankel's proof uses many techniques from several complex variables and does not extend to domains in a general flag manifold. However in the special case of $\Gr_p(\Rb^{2p})$ the following is known.

\begin{theorem}\cite{vLZ2015}
Suppose $p > 1$ and $\Omega  \subset \Gr_p(\Rb^{2p})$ is a convex divisible domain. Then $\Omega$ is projectively isomorphic to $\Bc_{p,p}$.
\end{theorem}

Based on these examples we conjecture (also see Conjecture 1.7 in~\cite{vLZ2015}):

\begin{conjecture} Suppose $G$ is a connected non-compact simple Lie group with trivial center and $P \leq G$ is a parabolic subgroup. If $G/P$ is not isomorphic to a real projective space, then every convex divisible domain in $G/P$ is homogeneous. 
\end{conjecture} 
 
\section{The automorphism group is closed}

It will be useful in what follows to know the automorphism group of an open domain is always closed. 

\begin{proposition}\label{prop:closed}
 Suppose $G$ is a connected semisimple Lie group with finite center. If $P \leq G$ is a parabolic subgroup and $\Omega \subset G/P$ is an open set, then
$\Aut(\Omega) \leq G$ is closed.
\end{proposition}

\begin{proof}
 Fix a distance $d$ on $G/P$ which is induced by a Riemannian metric. Suppose that a sequence $\varphi_n \in \Aut(\Omega)$ converges to some $\varphi \in G$. 
Now since 
\begin{align*}
 \{ \varphi_n : n \in \Nb\} \subset G
\end{align*}
 is relatively compact, there exists some $K \geq 1$ such that 
\begin{align*}
\frac{1}{K} d(x,y) \leq d(\varphi_n x, \varphi_n y) \leq K d(x,y)
\end{align*}
for all $x,y \in G/P$ and $n \in \Nb$.

Next define the function $\delta_\Omega: \Omega \rightarrow \Rb_{>0}$ by
\begin{align*}
\delta_\Omega(x) = \inf \{ d(x,y) : y \in G/P - \Omega\}.
\end{align*}

Now fix $x \in \Omega$, then $\delta_\Omega(\varphi_n x) \geq \frac{1}{K} \delta_\Omega(x)$ for all $n \in \Nb$. Since $d(\varphi x, \varphi_n x) < \frac{1}{K} \delta_\Omega(x)$ for $n$ large, we see that $\varphi x \in \Omega$. Since $x \in \Omega$ was arbitrary we see that $\varphi(\Omega) \subset \Omega$. Applying the same argument to the sequence 
$\varphi_n^{-1} \rightarrow \varphi^{-1}$ we see that $\varphi^{-1}(\Omega) \subset \Omega$. Thus $\varphi(\Omega) = \Omega$ and so $\varphi \in \Aut(\Omega)$. 
\end{proof}

\section{Parabolic subgroups}

For the rest of this section suppose that $G$ is a connected semisimple Lie group without compact factors and with trivial center. Let $K \leq G$ be a maximal compact subgroup. Then the manifold $X = G/K$ has a $G$-invariant non-positively curved Riemannian metric $g$ and $(X,g)$ is a symmetric space. 

Let $X(\infty)$ be the ideal boundary of  $X$. For a geodesic $\gamma: \Rb \rightarrow X$ let 
\begin{align*}
\gamma(\infty) = \lim_{t \rightarrow \infty} \gamma(t) \in X(\infty)
\end{align*}
and 
\begin{align*}
\gamma(-\infty) = \lim_{t \rightarrow- \infty} \gamma(t) \in X(\infty).
\end{align*}

\begin{definition} \ 
\begin{enumerate}
\item A subgroup $P \leq G$ is called \emph{parabolic} if $P$ is the stabilizer in $G$ of some point $x \in X(\infty)$. 
\item Two parabolic subgroups $P,Q \leq G$ are said to be \emph{opposite} if there exists a geodesic $\gamma:\Rb \rightarrow X$ such that $P$ is the stabilizer of $\gamma(\infty)$ and $Q$ is the stabilizer of $\gamma(-\infty)$. 
\end{enumerate}
\end{definition}

In this section we recall the basic properties of parabolic subgroups. We will mostly rely on the exposition in Eberlein's book on symmetric spaces~\cite{E1996} and Warner's book on harmonic analysis on semisimple groups~\cite{W1972}. 

\begin{theorem}\label{thm:para_1}
With the notation above:
\begin{enumerate}
\item There are only finitely many conjugacy classes of parabolic subgroups.
\item Suppose $P,Q_1,Q_2 \leq G$ are parabolic subgroups. If $Q_1,Q_2$ are both opposite to $P$, then there exists some $p \in P$ such that $pQ_1p^{-1} = Q_2$. 
\end{enumerate}
\end{theorem}

\begin{proof}
For the first assertion see Corollary 2.17.23 in~\cite{E1996}. The idea is: fix a maximal flat $F$ in $X$, then given a geodesic $\gamma: X \rightarrow \Rb$ there exists $g \in G$ such that $g \gamma \subset F$. Then using the root space decomposition of $\gL$, the Lie algebra of $G$, associated to $F$ one shows that only finitely many different groups arise as stabilizers of points $x \in F(\infty)$. 

Next suppose that $Q_1,Q_2$ are both opposite to $P$. For $i\in \{1,2\}$, let $\gamma_i: \Rb \rightarrow X$ be a geodesic where $P$ is the stabilizer of $\gamma_i(\infty)$ and $Q_i$ is the stabilizer of $\gamma_i(-\infty)$. Since $P$ acts transitively on $X$ (see Proposition 2.17.1 in~\cite{E1996}), there exists $p \in P$ such that $p\gamma_1(0)=\gamma_2(0)$. Then $pQ_1p^{-1}$ is the stabilizer of $p \gamma_1(-\infty)$. Now since $P$ is the stabilizer of $p\gamma_1(\infty)$ and $\gamma_2(\infty)$, Proposition 2.17.15 in~\cite{E1996} implies that $d(p)\gamma_1^\prime(0)$ and $\gamma_2^\prime(0)$ are in the same Weyl Chamber $W$ in $S_{\gamma_2(0)} X$ (the unit tangent sphere of $X$ above $\gamma_2(0)$). Finally, since $-W \subset S_{\gamma_2(0)}X$ is also a Weyl chamber we see that $-d(p)\gamma_1^\prime(0)$ and $-\gamma_2^\prime(0)$ are also in the same Weyl chamber. So appealing to Proposition 2.17.15 in~\cite{E1996} again, we see that $pQ_1p^{-1} = Q_2$. 
\end{proof}

\begin{corollary}\label{cor:para_1}
Suppose $P, Q \leq G$ are opposite parabolic subgroups. Then $gPg^{-1}$ and $hQh^{-1}$ are opposite if and only if $g^{-1}h \in PQ$. 
\end{corollary}

\begin{proof}
Clearly $gPg^{-1}$ and $hQh^{-1}$ are opposite if and only if $P$ and $(g^{-1}h)Q(g^{-1}h)^{-1}$ are opposite. But by Theorem~\ref{thm:para_1} this happens if and only if $pg^{-1}hQ(pg^{-1}h)^{-1}=Q$ for some $p \in P$. But parabolic subgroups are their own normalizer, see Proposition 2.17.25 in~\cite{E1996}, and so this happens if and only if $g^{-1}h \in PQ$. 
\end{proof}

Since $PQ \subset G$ is open and dense (see for instance~\cite[Proposition 1.2.4.10]{W1972}), we immediately see that:

\begin{corollary}\label{cor:opposite_open_dense}
Suppose $P, Q \leq G$ are opposite parabolic subgroups. Then the set
\begin{align*}
\Oc:=\left\{ (gP, hQ) : gPg^{-1} \text{ and } hQh^{-1} \text{ are opposite} \right\} \subset G/P \times G/Q
\end{align*}
is open and dense.
\end{corollary}

\begin{proposition}\label{prop:bigger_para}
Suppose $P,Q \leq G$ are opposite parabolic subgroups. If $P_0 \gneqq P$ is a parabolic subgroup, then there exists a unique parabolic subgroup $Q_0 \gneqq Q$ such that $P_0$ and $Q_0$ are opposite. Moreover:
\begin{enumerate}
\item If $gPg^{-1}$ and $hQh^{-1}$ are opposite, then $gP_0 g^{-1}$ and $hQ_0 h^{-1}$ are opposite.
\item There exists $w \in P_0$ such that $wPw^{-1}$ and $Q$ are not opposite. 
\end{enumerate}
\end{proposition}

\begin{proof}
Let $\gamma: \Rb \rightarrow X$ be a geodesic such that $P$ is the stabilizer of $\gamma(\infty)$ and $Q$ is the stabilizer of $\gamma(-\infty)$. Let $F$ be a maximal flat containing $\gamma$. Pick $x \in X(\infty)$ such that $P_0$ is the stabilizer of $x \in X(\infty)$. Then $x \in F(\infty)$ by Proposition 3.6.26 in~\cite{E1996}. Now using Proposition 2.17.13 in~\cite{E1996} we see that there exists a parabolic subgroup $Q_0 \geq Q$ such that $P_0$ and $Q_0$ are opposite. Moreover, if $Q_1 \geq Q$ is a parabolic subgroup and $Q_1$ is the stabilizer of $y \in X(\infty)$ then, by Proposition 3.6.26 in~\cite{E1996} again,  $y \in F(\infty)$. So using Proposition 2.17.13 in~\cite{E1996} again we see that $Q_0$ is unique. 

By Corollary~\ref{cor:para_1}, if $gPg^{-1}$ and $hQh^{-1}$ are opposite, then $g^{-1}h \in PQ$, then $g^{-1}h \in P_0 Q_0$, then $gP_0 g^{-1}$ and $hQ_0 h^{-1}$ are opposite.

As before let $F$ be a maximal flat containing $\gamma$. Then $F(\infty)$ decomposes into Weyl chambers and if $P_z$ is the stabilizer of $z \in F(\infty)$ then $P_z$ is opposite to $Q$ if and only if $z$ and $\gamma(\infty)$ are in the same Weyl chamber (by Proposition 2.17.13 in~\cite{E1996}). Let $K$ be the stabilizer of $\gamma(0)$ and let 
\begin{align*}
W = \{ k \in K : k F = F\}.
\end{align*}
Now by Lemma 1.2.4.6 in~\cite{W1972} there exists $w \in W \cap P_0$ such that $w \gamma(\infty)$ is not in the same Weyl chamber as $\gamma(\infty)$. Hence $wPw^{-1}$ is not opposite to $Q$.
\end{proof}

\subsection{Representations}

Using the theory of irreducible representations of reductive groups {Gu{\'e}ritaud} et al.~\cite[Lemma 4.5, Proposition 4.6]{GGKW2015b} proved the following theorem.

\begin{theorem}\label{thm:repn}
Suppose $P,Q \leq G$ are opposite parabolic subgroups. Then there exists an real vector space $V$, an irreducible representation $\tau: G \rightarrow \PGL(V)$, a line $\ell \subset V$, and a hyperplane $H \subset V$ such that:
\begin{enumerate}
\item $\ell + H = V$.
\item The stabilizer of $\ell$ in $G$ is $P$ and the stabilizer of $H$ in $G$ is $Q$. 
\item $gPg^{-1}$ and $hQ h^{-1}$ are opposite if and only if $\tau(g)\ell$ and $\tau(g)H$ are transverse. 
\end{enumerate}
\end{theorem}

Now fix $P,Q \leq G$ opposite parabolic subgroups. Let $\tau: G \rightarrow \PGL(V)$ be an irreducible representation, $\ell \subset V$ a line, and $H \subset V$ a hyperplane as in Theorem~\ref{thm:repn}. Fix some $x_0 \in V$ such that $\Rb x_0 = \ell$ and fix some $f_0 \in V^*$ a functional with $\ker f_0 = H$. Consider the dual representation $\tau^*:G \rightarrow \PGL(V^*)$. Now define the maps $\iota: G/P \rightarrow \Pb(V)$ and $\iota^*: G/Q \rightarrow \Pb(V^*)$ by
\begin{align*}
\iota(gP) = [\tau(g)x_0] \text{ and }\iota^*(gQ) = [\tau^*(g)f_0].
\end{align*}

It will be helpful to observe the following:

\begin{lemma}\label{lem:spanning}
With the notation above, if $\Oc \subset G/P$ is any open set, then there exists $x_1, \dots, x_D \in \Oc$ such that 
\begin{align*}
\iota(x_1) \oplus \dots \oplus \iota(x_D) = V.
\end{align*}
Analogously, if $\Oc \subset G/Q$ is any open set, then there exists $\xi_1, \dots, \xi_D \in \Oc$ such that 
\begin{align*}
\iota^*(\xi_1) \oplus \dots \oplus \iota^*(\xi_D) = V^*.
\end{align*}
\end{lemma}

\begin{proof}
We can identify $V$ with $\Rb^D$. Consider the map $\Phi: G^D \rightarrow \wedge^D \Rb^D$ given by
\begin{align*}
\Phi(g_1, \dots, g_D)=  (\tau(g_1) x_0) \wedge ( \tau(g_2) x_0 ) \wedge \dots  \wedge ( \tau(g_D) x_0 ).
\end{align*}
Since $\tau$ is an irreducible representation there exists $(g_1, \dots, g_D) \in G^D$ such that $\Phi(g_1, \dots, g_D) \neq 0$. Then since $\Phi$ is real analytic for any open set $\Oc \subset G^D$ there exists $(h_1, \dots, h_D) \in \Oc$ such that $\Phi(h_1, \dots, h_D) \neq 0$. This implies the first assertion of the Lemma. The second assertion has the exact same proof.
\end{proof}

\begin{lemma}\label{lem:embedding} The maps $\iota: G/P \rightarrow \Pb(V)$ and $\iota^*:G/Q \rightarrow \Pb(V^*)$ are embeddings. \end{lemma}

\begin{proof} The maps are smooth, injective, and have constant rank (due to the $G$ action). Further $G/P$ and $G/Q$ are compact, so the maps must be embeddings. \end{proof}

\subsection{Affine charts}\label{subsec:affine_chart}

\begin{definition}
Suppose that $P \leq G$ is a parabolic subgroup. Then a non-empty subset $\Ab \subset G/P$ is called an \emph{affine chart} if there exists a parabolic subgroup $Q \leq G$ such that
\begin{align*}
\Ab = \{ x=gP \in G/P : gPg^{-1} \text{ is opposite to } Q \}.
\end{align*}
\end{definition}

Let $P,Q \leq G$ be opposite parabolic subgroups. Let $N \leq Q$ be the unipotent radical of $Q$, then $QP=NP$ by Proposition 2.17.5 and 2.17.13 in~\cite{E1996}. Now suppose $\Ab \subset G/P$ is an affine chart and 
\begin{align*}
\Ab = \{ x=gP \in G/P : gPg^{-1} \text{ is opposite to } Q^\prime \}
\end{align*}
for some parabolic subgroup $Q^\prime \leq G$. Since $\Ab$ is non-empty, $g_0P \in \Ab$ for some $g_0 \in G$. So $P$ is opposite to $g_0^{-1} Q^\prime g_0$. Now by Theorem~\ref{thm:para_1}, $g_0^{-1}Q^\prime g_0= p Q p^{-1}$ for some $p \in P$. Then 
\begin{align*}
\Ab &= (g_op)^{-1} \{ gP \in G/P : gPg^{-1} \text{ is opposite to } Q \} \\
& = (g_op)^{-1} \{ gP \in G/P : g \in QP \} \\
& = (g_0p)^{-1} NP.
\end{align*}
In particular, an affine chart in $G/P$ is a translate of the big Bruhat cell.

\section{A Carath{\'e}odory type metric}\label{sec:cara}

Suppose $G$ is a non-compact connected simple Lie group with trivial center and $P \leq G$ is a parabolic subgroup. Fix a parabolic subgroup $Q \leq G$ opposite to $P$. 

Given a set $\Omega \subset G/P$ define
\begin{align*}
\Omega^* := \{ \xi \in G/Q:  \xi \text{ is opposite to every } x \in \Omega \}.
\end{align*}
Also given $\xi \in G/Q$ define (as before)
\begin{align*}
Z_\xi := \{ x \in G/P : x \text{ is not opposite to } \xi \}.
\end{align*}
Then $\xi \in \Omega^*$ if and only if $Z_\xi \cap \Omega = \emptyset$. 

\begin{lemma}\label{lem:dual} \ \begin{enumerate}
\item If $\Omega$ is open, then $\Omega^*$ is compact.
\item $\Omega$ is bounded in an affine chart if and only $\Omega^*$ has non-empty interior. 
\end{enumerate} 
\end{lemma}

\begin{proof}
Suppose that $\Omega$ is open and $\xi_n$ is a sequence in $\Omega^*$ converging to some $\xi \in G/Q$. Suppose for a contradiction that $\xi \notin \Omega^*$. Then there is some $x \in \Omega$ such that $x$ is not opposite to $\xi$. Since $\xi_n \rightarrow \xi$, we can find $g_n \in G$ converging to $\id \in G$ such that $g_n \xi = \xi_n$. Then $\xi_n$ is not opposite to $g_n x$. But since $g_n \rightarrow \id$ and $\Omega$ is open, we see that $g_n x \in \Omega$ for large $n$. So we have a contradiction. 

Now suppose that $\Omega$ is bounded in some affine chart $\Ab \subset G/P$. Then $\Ab = G/P - Z_\eta$ for some $\eta \in G/Q$. Since $\Omega$ is bounded in $\Ab$, 
\begin{align*}
\overline{\Omega} \cap Z_\eta = \emptyset.
\end{align*}
Suppose for a contradiction that $\eta$ is not in the interior of $\Omega^*$. Then there exists $\eta_n \in G/Q$ such that $\eta_n \rightarrow \eta$ and $x_n \in \Omega$ such that $x_n$ is not opposite to $\eta_n$. By passing to a subsequence we can suppose that $x_n \rightarrow x \in \overline{\Omega}$. But then, by Corollary~\ref{cor:opposite_open_dense}, we see that $x$ is not opposite to $\eta$ which is a contradiction.  

Now suppose that $\eta$ is in the interior of $\Omega^*$. Since $\eta \in\Omega^*$ we see that $\Omega$ is contained in the affine chart $\Ab = G/P - Z_\eta$. Now fix an neighborhood $U$ of the identity in $G$ such that $U \eta \subset \Omega^*$. Then 
\begin{align*}
\Omega \cap ( U \cdot Z_{\eta} ) = \emptyset
\end{align*}
and $U \cdot Z_\eta$ is a neighborhood of $Z_\eta$. Thus $\Omega$ is bounded in $\Ab$. 
\end{proof}

Now define maps $\iota: G/P \rightarrow \Pb(V)$ and $\iota^*: G/Q \rightarrow \Pb(V^*)$ as in the discussion following Theorem~\ref{thm:repn}.

Then for a proper domain  $\Omega \subset G/P$ define the function 
\begin{align*}
C_\Omega:\Omega \times \Omega \rightarrow \Rb_{\geq 0}
\end{align*}
by 
\begin{align*}
C_\Omega(x,y) = \sup_{\xi, \eta \in \Omega^*}\log \abs{ \frac{ \iota^*(\xi)\Big(\iota(x)\Big)\iota^*(\eta)\Big(\iota(y)\Big)}{\iota^*(\xi)\Big(\iota(y)\Big)\iota^*(\eta)\Big(\iota(x)\Big)}}.
\end{align*}

\begin{theorem}\label{thm:cara_metric}
Suppose $\Omega \subset G/P$ is a proper domain. Then $C_{\Omega}$ is a $\Aut(\Omega)$-invariant metric on $\Omega$ which generates the standard topology. 
\end{theorem}

\begin{proof}[Proof of Theorem~\ref{thm:cara_metric}]
Since $\Aut(\Omega)$ preserves $\Omega^*$ we see that $C_\Omega$ is $\Aut(\Omega)$-invariant. 

Fix $x,y \in \Omega$ distinct. We will show that $C_{\Omega}(x,y) > 0$. Since $\Omega^* \subset G/Q$ has non-empty interior, Lemma~\ref{lem:spanning} implies that the set $\iota^*(\Omega^*) \subset \Pb(V^*)$ contains a basis of $V^*$. So we can find a basis $f_1, \dots, f_D \in V^*$ such that $[f_1], \dots, [f_D] \in \iota^*(\Omega^*)$. Now let $e_1, \dots, e_d$ be the basis of $V$ dual to $f_1, \dots, f_D$, that is 
\begin{align*}
f_j(e_i) = \left\{ \begin{array}{ll} 
1 & \text{ if } i =j \\
0 & \text{ otherwise.} 
\end{array} \right.
\end{align*}
Then $\iota(x) = [\sum x_i e_i]$ and $\iota(y) = [\sum y_i e_i]$ for some $x_i, y_i \in \Rb$. Since $\iota: G/P \rightarrow \Pb(V)$ is injective there exists $i,j$ such that $x_i / x_j \neq y_i / y_j$. Then 
\begin{align*}
C_\Omega(x,y) \geq \abs{ \log \abs{ \frac{f_i(\iota(x) )f_j(\iota(y) )}{f_i(\iota(y) )f_j(\iota(x) )} }}= \abs{ \log \abs{ \frac{x_i y_j}{x_j y_i}} }> 0.
\end{align*}

Next fix $x,y,z \in \Omega$. Since $\Omega^*$ is compact there exists $\xi, \eta \in \Omega^*$ such that 
\begin{align*}
C_{\Omega}(x,y) 
&= \log \abs{ \frac{ \iota^*(\xi)\Big(\iota(x)\Big)\iota^*(\eta)\Big(\iota(y)\Big)}{\iota^*(\xi)\Big(\iota(y)\Big)\iota^*(\eta)\Big(\iota(x)\Big)}}.
\end{align*}
Then
\begin{align*}
C_{\Omega}(x,y) 
& =   \log \abs{ \frac{ \iota^*(\xi)\Big(\iota(x)\Big)\iota^*(\eta)\Big(\iota(z)\Big)}{\iota^*(\xi)\Big(\iota(z)\Big)\iota^*(\eta)\Big(\iota(x)\Big)}} +  \log \abs{ \frac{ \iota^*(\xi)\Big(\iota(z)\Big)\iota^*(\eta)\Big(\iota(y)\Big)}{\iota^*(\xi)\Big(\iota(y)\Big)\iota^*(\eta)\Big(\iota(z)\Big)}}. \\
& \leq C_{\Omega}(x,z) + C_{\Omega}(z,y).
\end{align*}
So $C_{\Omega}$ satisfies the triangle inequality and hence is a metric.

It remains to show that $C_\Omega$ generates the standard topology. Since $\Omega^*$ is compact, $C_\Omega$ is continuous with respect to the standard topology on $\Omega$. Thus to show that $C_{\Omega}$ generates the standard topology it is enough to show: for any $x_0 \in \Omega$ and $U \subset \Omega$ an open neighborhood of $x_0$ there exists $\delta > 0$ such that 
\begin{align*}
\{ y \in \Omega : C_{\Omega}(x_0,y) < \delta \} \subset U.
\end{align*}
As before fix a basis $f_1, \dots, f_D \in V^*$ such that $[f_1], \dots, [f_D] \in \iota^*(\Omega^*)$. Then let $e_1, \dots, e_D$ be the basis of $V$ dual to $f_1, \dots, f_D$. Since $[f_1] \in \Omega^*$, $\iota(x_0)$ has non-zero $e_1$ component. Then 
\begin{align*}
\iota(x_0) = \left[ e_1 + \sum_{i = 2}^D x_i e_i\right]
\end{align*}
for some $x_2, \dots, x_D \in \Rb$. Since $[f_i] \in \Omega^*$, $x_i \neq 0$. Then since $U$ is open and $\iota : G/P \rightarrow \Pb(V)$ is an embedding, see Lemma~\ref{lem:embedding}, there exists $\epsilon > 0$ such that
\begin{align*}
\left\{ \left[ e_1 + \sum_{i =2}^D y_i e_i\right]  : e^{-\epsilon} < \abs{x_i / y_i} < e^{\epsilon} \right\} \cap \iota(G/P) \subset \iota(U).
\end{align*}
Moreover, if $\iota(y) = \left[ e_1 + \sum_{i = 2}^D y_i e_i\right]$ then
\begin{align*}
C_\Omega(x,y) \geq \max_{2 \leq i \leq D}  \abs{\log \abs{\frac{x_i}{y_i}}}.
\end{align*}
So if $\delta = \epsilon$ we see that 
\begin{align*}
\{ y \in \Omega : C_{\Omega}(x_0,y) < \delta \} \subset U.
\end{align*}
Thus $C_\Omega$ generates the standard topology. 
\end{proof}

The existence of an invariant metric implies the following:

\begin{corollary}\label{cor:proper}
Suppose $\Omega \subset G/P$ is a proper domain. Then the action of $\Aut(\Omega)$ on $\Omega$ is proper. 
\end{corollary}

The proof of Corollary~\ref{cor:proper} is essentially the proof of~\cite[Proposition 4.4]{Z2015b} taken verbatim. 

\begin{proof}
This argument requires some care because $C_{\Omega}$ may not be a complete metric (see Theorem~\ref{thm:completeness}). Fix a compact set 
$K \subset \Omega$, we claim that
\begin{align*}
\{ \varphi \in \Aut(\Omega) : \varphi K \cap K \neq \emptyset \} \subset \Aut(\Omega)
\end{align*}
is compact. So suppose that $\varphi_n k_n \in K$ for some sequences $\varphi_n \in \Aut(\Omega)$ and $k_n \in K$. By passing to a 
subsequence we can suppose that $k_n \rightarrow k_\infty \in K$. Now since $C_{\Omega}$ is a locally compact metric (it generates the standard topology) 
and $K \subset \Omega$ is compact there exists some $\delta > 0$ such that the set 
\begin{align*}
K_1 = \{ q \in \Omega: C_{\Omega}(K,q) \leq 2\delta\}
\end{align*}
is compact. Next let 
\begin{align*}
K_2 = \{ q \in \Omega : C_{\Omega}(k_\infty,q) \leq \delta\}.
\end{align*}
Then for large $n$ we have $\varphi_n(K_2) \subset K_1$. Since $\varphi_n$ preserves the metric $C_{\Omega}$ and $(K_1, C_{\Omega}|_{K_1})$ is a complete metric space we can pass to a subsequence and assume that $\varphi_n|_{K_2}$ converges uniformly to a function $f:K_2 \rightarrow K_1$. Moreover
\begin{align*}
C_{\Omega}(f(p_1), f(p_2)) = \lim_{n \rightarrow \infty} C_\Omega(\varphi_n p_1, \varphi_n p_2) = C_{\Omega}(p_1,p_2)
\end{align*}
for all $p_1, p_2 \in K_2$. Since $C_\Omega$ is a metric generating the standard topology on $\Omega$, we see that $f$ induces a homeomorphism $K_2 \rightarrow f(K_2)$. 

Since $G$ is a simple Lie group with trivial center, the map $\tau: G \rightarrow \PGL(V)$ is proper. Now fix a norm on $V$ and an associated operator norm on $\GL(V)$. Next let $T_n \in \GL(V)$ be a representative of $\tau(\varphi_n)$ with $\norm{T_n}=1$. Then pass to a subsequence such that $T_n \rightarrow T$ in $\End(V)$. Now let $Y = (\iota)^{-1}(\ker T)$. Notice: if $x \in K_2 \cap (\Omega \setminus Y)$ then 
\begin{align*}
\iota(f(x)) = \lim_{n \rightarrow \infty} \iota( \varphi_n x) = \lim_{n \rightarrow \infty} \tau(\varphi_n)\iota( x) =T(\iota(x)).
\end{align*}
Now by Lemma~\ref{lem:spanning}, $\Omega \setminus Y$ is open and dense in $\Omega$. So $K_2 \cap (\Omega \setminus Y)$ has non-empty interior, so 
\begin{align*}
f \left( K_2 \cap (\Omega \setminus Y) \right) 
\end{align*} 
has non-empty interior. But then by Lemma~\ref{lem:spanning}, the image of $T$ contains a spanning set of $V$. Thus $T \in \GL(V)$. Thus, since $\tau$ is a proper map, we can pass to a subsequence such that $\varphi_n$ converges to some  $\varphi$ in $G$. Then since $\Aut(\Omega)$ is closed we see that $\varphi \in \Aut(\Omega)$. 

Thus the set 
\begin{align*}
\{ \varphi \in \Aut(\Omega) : \varphi K \cap K \neq \emptyset \} \subset \Aut(\Omega)
\end{align*}
is compact. So $\Aut(\Omega)$ acts properly on $\Omega$.
\end{proof}

\subsection{The Hilbert metric}\label{subsec:Hilbert} In this subsection we compare the metric $C_\Omega$ to the Hilbert metric. 

\begin{definition}
An open set $\Cc \subset \Pb(\Rb^d)$ is called \emph{properly convex} if for every projective line $\ell \subset \Pb(\Rb^d)$ the intersection $\ell \cap \Cc$ is connected and $\overline{\ell \cap \Cc} \neq \ell$. 
\end{definition}

Suppose $\Cc$ is properly convex. Given two points $x,y \in \Cc$ let $\ell_{xy}$ be a projective line containing $x$ and $y$. Then the \emph{Hilbert distance} between them is defined to be 
\begin{align*}
H_{\Cc}(x,y) = \log \frac{ \abs{y-a}\abs{x-b}}{\abs{x-a}\abs{y-b}}
\end{align*}
where $\{a,b\} = \partial \Cc \cap \ell_{xy}$ and we have the ordering $a,x,y,b$ along $\ell_{xy}$. Now define 
\begin{align*}
\Cc^{dual} = \{ [f] \in \Pb(\Rb^{d*}) :  f(x) \neq 0 \text{ for all } x \in \Cc  \}.
\end{align*}
By the supporting hyperplane definition of convexity, we see that every $a \in \partial \Cc$ is contained in the kernel of some $f \in \Cc^{dual}$. Then it is not hard to show that
\begin{align*}
H_{\Cc}(x,y) =  \sup_{f,g \in \Cc^{dual}} \log  \abs{ \frac{f(y)g(x)}{f(x)g(y)}}.
\end{align*}

Let $e_1, \dots, e_d$ be the standard basis of $\Rb^d$. Let $P \leq \PGL_d(\Rb)$ be the stabilizer of the line $\Rb e_1$ and $Q \leq \PGL_d(\Rb)$ the stabilizer of $\Rb e_2 + \dots + \Rb e_d$. Then $P,Q$ are opposite parabolic subgroups. We can take $\tau: \PGL_{d}(\Rb) \rightarrow \PGL_d(\Rb)$ to be the identity representation. Then $\iota$ identifies $G/P$ with $\Pb(\Rb^d)$ and $\iota^*$ identifies $G/Q$ with $\Pb(\Rb^{d*})$. 

Suppose $\Omega \subset G/P$ is a proper dual convex domain. Then $\iota(\Omega) \subset \Pb(\Rb^d)$ is a properly convex set. Moreover, 
\begin{align*}
\iota^*(\Omega^*) = \{ f \in \Pb(\Rb^*) : f(x) \neq 0 \text{ for all } x \in \iota(\Omega) \} = \iota(\Omega)^{dual}
\end{align*}
Thus
\begin{align*}
C_\Omega(x,y) = H_{\iota(\Omega)}(\iota(x), \iota(y))
\end{align*}
for all $x,y \in \Omega$. 

 \subsection{The Carath{\'e}odory and Kobayashi metric} For domains $\Oc \subset \Pb(\Rb^d)$ which are not convex, Kobayashi~\cite{K1977} constructed two invariant metrics using projective maps to and from the unit interval. Let
\begin{align*}
I : = \{ [1:x] : \abs{x} < 1\} \subset \Pb(\Rb^2).
\end{align*}
For two open sets $\Omega_1 \subset \Rb(\Rb^{d_1+1})$ and $\Omega_2 \subset \Pb(\Rb^{d_2+1})$ let $\Proj(\Omega_1, \Omega_2)$ be the space of maps $f:\Omega_1 \rightarrow \Omega_2$ such that $f = T|_{\Omega_1}$ for some $T \in \Pb(\Lin(\Rb^{d_1+1}, \Rb^{d_2+1}))$ with $\ker T \cap \Omega_1 = \emptyset$.

For a domain $\Oc \subset \Pb(\Rb^d)$ define the two quantities:
\begin{align*}
c_{\Oc}(x,y) = \sup \left\{ H_{I}(f(x),f(y)) :  f \in \Proj(\Omega, I)  \right\},
\end{align*}
and
\begin{align*}
L_{\Oc}(x,y) = \inf \left\{ H_I(u,w) : f \in \Proj(I,\Omega) \text{ with } f(u)=x \text{ and } f(w)=y\right\}.
\end{align*}
The function $c_{\Oc}$ always satisfies the triangle inequality, but $L_{\Oc}$ may not. So we introduce:
\begin{align*}
k_{\Oc}(x,y) = \inf \left\{ \sum_{i=0}^N L_{\Oc}(x_i, x_{i+1}) : N > 0, \  x_0, x_1, \dots, x_{N+1} \in \Oc, \ x = x_0, \ y = x_{N+1} \right\}.
\end{align*}

Kobayashi then proved:

\begin{proposition}\cite{K1977}
Suppose $\Oc \subset \Pb(\Rb^{d+1})$ is a proper domain. Then $c_{\Oc}$ and $k_{\Oc}$ are $\Aut(\Oc)$-invariant metrics on $\Oc$. Moreover, if $\Oc$ is convex, then $c_{\Oc} = k_{\Oc} = H_{\Oc}$. 
\end{proposition}

As in Subsection~\ref{subsec:Hilbert} let $G = \PGL_d(\Rb)$,  $P$ the stabilizer of $\Rb e_1$, $Q$ the stabilizer of $\Rb e_2 + \dots + \Rb e_d$, $\tau: \PGL_d(\Rb) \rightarrow \PGL_d(\Rb)$ the identity representation, and $\iota, \iota^*$ the induced maps. 

\begin{proposition}
With the notation above, suppose $\Omega \subset G/P$ is a proper domain. Then $C_\Omega(x,y) = c_{\iota(\Omega)}(\iota(x), \iota(y))$ for all $x,y \in \Omega$.
\end{proposition}

\begin{proof}
View $\Omega$ as a subset of $\Pb(\Rb^{d})$ and $\Omega^*$ as a subset of $\Pb(\Rb^{d*})$. Then for $\xi_1, \xi_2 \in \Omega^*$ distinct there is a map $T \in \Pb(\Lin(\Rb^{d+1},\Rb^2))$ with $T^{-1}( [1:-1]) = \ker \xi_1$ and $T^{-1}([1:1])= \ker \xi_2$, and $T(\Omega) \subset I$. Then $f = T|_{\Omega} \in \Proj(\Omega, I)$ and it is straightforward to show that 
\begin{align*}
\abs{ \log  \abs{ \frac{\xi_1(x) \xi_2(y)}{\xi_1(y)\xi_2(x)}} } = H_{I}(f(x),f(y)).
\end{align*}
Conversely, if $f=T|_{\Omega} \in \Proj(\Omega, I)$ then $T^{-1}([1:-1]) = \ker \xi_1$ and $T^{-1}([1:1]) = \ker \xi_2$ for some $\xi_1, \xi_2 \in \Omega^*$. Then as before
\begin{equation*}
\abs{ \log  \abs{ \frac{\xi_1(x) \xi_2(y)}{\xi_1(y)\xi_2(x)}} } = H_{I}(f(x),f(y)). \qedhere
\end{equation*}
\end{proof}

This proposition shows that $C_\Omega$ can be seen as an analogue of the Carath{\'e}odory metric from several complex variables. 

For certain flag manifolds $G/P$ it is also possible to construct an invariant metric $K_\Omega$ using projective maps of $\Pb(\Rb^2)$ into $G/P$. See~\cite[Section 4]{vLZ2015} for the case when $G/P = \Gr_p(\Rb^{p+q})$. 

\section{Proof of Theorem~\ref{thm:proper_non_maximal}}\label{sec:main_thm}

Assume for a contradiction that there exists $G$ a non-compact connected simple Lie group with trivial center, $P \leq G$ a non-maximal parabolic subgroup, 
and $\Omega \subset G/P$ a proper quasi-homogeneous domain. Let $K \subset \Omega$ be a compact subset such that $\Aut(\Omega) \cdot K = \Omega$.

Fix some parabolic subgroup $P_0 \gneqq P$. Next consider the natural projection $\pi:G/P \rightarrow G/P_0$. Let $\Omega_0 = \pi(\Omega)$. Since $\pi(gx) = g\pi(x)$ we see that 
$\Aut(\Omega) \leq \Aut(\Omega_0)$ and 
\begin{align*}
 \Aut(\Omega) \cdot \pi(K) = \pi( \Aut(\Omega) \cdot K) = \pi(\Omega) = \Omega_0.
\end{align*}
Hence $\Omega_0$ is quasi-homogeneous. 

Next fix a parabolic subgroup $Q \leq G$ opposite to $P$. By Proposition~\ref{prop:bigger_para} there exists a unique parabolic subgroup $Q_0 \geq Q$ which is opposite to $P_0$. As in Section~\ref{sec:cara} we can define domains $\Omega^* \subset G/Q$ and $\Omega^*_0 \subset G/Q_0$. Moreover, if $\pi^*: G/Q \rightarrow G/Q_{0}$ is the natural projection, we see from 
Proposition~\ref{prop:bigger_para} that $\pi^*(\Omega^*) \subset \Omega_0^*$. Since $\pi^*$ is an open map, $\Omega_0^*$ has non-empty interior and thus by Lemma~\ref{lem:dual} we see that $\Omega_0$ is a proper domain. Thus by Corollary~\ref{cor:proper} we see that $\Aut(\Omega)$ acts properly on $\Omega_0$. 

Now fix some point $gP_0 \in \Omega_0$ and let
\begin{align*}
\Omega_g:= \pi_1^{-1}(gP_0) \cap \Omega =gP_0/P \cap \Omega.
\end{align*}
By Proposition~\ref{prop:bigger_para} part (2) the set $gP_0/P$ is not contained in any affine chart of $G/P$. In particular, $\Omega_g \neq gP_0/P$. So we can pick $z_n \in \Omega_g$ such that $z_n \rightarrow z \in \partial \Omega$. Next pick $\varphi_n \in \Aut(\Omega)$ such that $z_n \in \varphi_n K$. Since $\Aut(\Omega)$ is closed and $z \in \partial \Omega$ we see that $\varphi_n \rightarrow \infty$ in $G$. Then
\begin{align*}
 \varphi_n^{-1} (gP_0) = \pi(\varphi_n^{-1} z_n) \in \pi(K).
\end{align*}
Since $\Aut(\Omega)$ acts properly on $\Omega_0$ we have a contradiction.

\section{The general semisimple case}\label{sec:products} 

For the rest of this section, suppose $G$ is a connected semisimple Lie group with trivial center and no compact factors, $P \leq G$ is a parabolic subgroup, and $\Omega \subset G/P$ is a proper domain. 

Then there exists $G_1, \dots, G_r$ non-compact simple Lie groups each with trivial centers such that
\begin{align*}
G \cong \prod_{i=1}^r G_i.
\end{align*}
Further, we can find subgroups $P_i \leq G_i$ such that $P \cong \prod_{i=1}^r P_i$. Moreover, either $P_i = G_i$ or $P_i \leq G_i$ is a parabolic subgroup. As in the statement of Theorem~\ref{thm:products} we assume that $P_i \neq G_i$ for all $1 \leq i \leq r$.

For $1 \leq i \leq r$ let $\pi_i : G/P \rightarrow G_i/P_i$  and $\rho_i: G \rightarrow G_i$ be the natural projections. Next let $\Omega_i = \pi_i(\Omega)$. Then (by definition) $\Omega_i$ is a proper domain. Moreover, $\rho_i(\Aut(\Omega)) \subset \Aut(\Omega_i)$ and so Corollary~\ref{cor:proper} implies that $\rho_i(\Aut(\Omega))$ acts properly on $\Omega_i$. Then since $\Omega$ is a subset of $\prod_{i=1}^r \Omega_i$:

\begin{proposition}\label{prop:proper_ss}
 $\Aut(\Omega)$ acts properly on $\Omega$. 
\end{proposition}

We now prove Theorem~\ref{thm:products}:

\begin{proof}[Proof of Theorem~\ref{thm:products}]
Let $K \subset \Omega$ be a compact subset such that $\Aut(\Omega) \cdot K = \Omega$. 

Next let
\begin{align*}
\wh{\Omega} := \prod_{i=1}^r \Omega_i \subset G/P.
\end{align*}
Since $\rho_i(\Aut(\Omega))$ acts properly on $\Omega_i$ we see that 
\begin{align*}
\prod_{i=1}^r \rho_i(\Aut(\Omega))
\end{align*}
acts properly on $\wh{\Omega}$. Then, since 
\begin{align*}
\Aut(\Omega)\leq \prod_{i=1}^r \rho_i(\Aut(\Omega))
\end{align*}
we see that $\Aut(\Omega)$ acts properly on $\wh{\Omega}$. 

Now assume that $\Omega \neq \wh{\Omega}$. Then there exists some $x \in \partial \Omega \cap \wh{\Omega}$. Pick $x_n \in \Omega$ such that $x_n \rightarrow x$. Then there exists $\varphi_n \in \Aut(\Omega)$ such that $x_n \in \varphi_n K$. 

Let $K^\prime \subset \wh{\Omega}$ be a compact neighborhood of $x$. Then for $n$ large 
\begin{align*}
\varphi_n K \cap K^\prime \neq \emptyset. 
\end{align*}
Then since $\Aut(\Omega)$ acts properly on $\wh{\Omega}$ we see that the set $\{ \varphi_n : n \in \Nb\} \subset G$ is relatively compact. Then, since $\Aut(\Omega)$ is closed, we can pass to a subsequence such that $\varphi_n \rightarrow \varphi \in \Aut(\Omega)$. Next let $k_n = \varphi_n^{-1} x_n$. By passing to another subsequence we can suppose that $k_n \rightarrow k \in K$. But then 
\begin{align*}
x = \varphi k \in \Omega
\end{align*}
and so we have a contradiction. 
\end{proof}

\section{Proof of Theorem~\ref{thm:covering_map}}\label{sec:covering_map}

 Suppose $M$ is a compact manifold and $G$ is a connected semi-simple Lie group with trivial center and no compact factors. Let $P \leq G$ be a parabolic subgroup and let $\{ (U_\alpha, \varphi_\alpha)\}_{\alpha \in \Ac}$ be a $(G,G/P)$-structure on $M$ such that the image of the developing map is bounded in an affine chart.

Now $\Omega:=\dev(\wt{M})$ is a proper quasi-homogeneous domain. By Proposition~\ref{prop:closed} and Proposition~\ref{prop:proper_ss}, $\Aut(\Omega)$ is a Lie group which acts properly on $\Omega$. Now, by a result of Palais~\cite{P1961}, there exists a $\Aut(\Omega)$-invariant Riemannian metric on $\Omega$. 

At this point the rest of the argument follows the proof of Proposition 3.4.10 in~\cite{T1997} verbatim, but we will provide the details for the reader's convenience. 

Since the Riemannian metric on $\Omega$ is $\Aut(\Omega)$-invariant, we can pull it back to a Riemannian metric $\wt{M}$ which descends to a Riemannian metric on $M$. Then the developing map $\dev:\wt{M} \rightarrow \Omega$ is a local isometry relative to the Riemannian metrics. 

Let $\overline{B}_{\wt{M}}(x,r)$ be the closed metric ball centered at $x$ of radius $r$ in $\wt{M}$ and let $\overline{B}_{\Omega}(y,r)$ be the closed metric ball centered at $y$ of radius $r$ in $\Omega$ (relative to the Riemannian metrics). 

Since $M$ is compact, there exists some $\epsilon > 0$ such that every metric ball of radius $\epsilon$ in $\wt{M}$ is convex (that is, every two points in the metric ball are joined by a unique geodesic in the metric ball). Since $\Aut(\Omega)$ acts co-compactly on $\Omega$, by possibly shrinking $\epsilon$ we can also assume that every metric ball of radius $\epsilon$ in $\Omega$ is convex. 

Then for any $x \in \wt{M}$ the developing map restricted to $\overline{B}_{\wt{M}}(x,\epsilon)$ is a homeomorpism onto its image: if $\dev(x_1) = \dev(x_2)$ for some $x_1, x_2 \in \overline{B}_{\wt{M}}(x,\epsilon)$ distinct, then the geodesic joining $x_1$ to $x_2$ in $\overline{B}_{\wt{M}}(x,\epsilon)$ is mapped to a self intersecting geodesic in $\overline{B}_{\Omega}(\dev(x),\epsilon)$ which is a contradiction. This implies that the developing map induces an isometry between $\overline{B}_{\wt{M}}(x,\epsilon)$ and $\overline{B}_{\Omega}(\dev(x),\epsilon)$. 

Now consider $y \in \Omega$ and $x \in \dev^{-1}(\overline{B}_\Omega(y,\epsilon/2))$. Then the metric ball $\overline{B}_{\wt{M}}(x,\epsilon)$ maps isometrically into $\Omega$ and hence properly contains a copy of $\overline{B}_\Omega(y,\epsilon/2)$. Thus the entire inverse image $ \dev^{-1}(\overline{B}(y,\epsilon/2))$ is a disjoint union of such homeomorphic copies. Therefore the developing map is a covering map.

\section{Dual convexity and completeness}\label{sec:complete}

Suppose $G$ is a non-compact connected simple Lie group with trivial center and $P \leq G$ is a parabolic subgroup. Fix a parabolic subgroup $Q \leq G$ opposite to $P$. 

Let $\Omega \subset G/P$ be a proper domain and define (as before)
\begin{align*}
\Omega^* := \{ \xi \in G/Q:  \xi \text{ is opposite to every } x \in \Omega \}.
\end{align*}
Also given $\xi \in G/Q$ define (as before)
\begin{align*}
Z_\xi := \{ x \in G/P : x \text{ is not opposite to } \xi \}.
\end{align*}
Then $\xi \in \Omega^*$ if and only if $Z_\xi \cap \Omega = \emptyset$. 

Now define maps $\iota: G/P \rightarrow \Pb(V)$ and $\iota^*: G/Q \rightarrow \Pb(V^*)$ as in the discussion following Theorem~\ref{thm:repn}. And let $C_\Omega$ be the metric on $\Omega$ constructed in Section~\ref{sec:cara}. 

We now provide a characterization of the domains where $C_\Omega$ is complete using the concept of dual convexity introduced in Definition~\ref{defn:dual_convex}. 

\begin{theorem}\label{thm:completeness} 
With the notation above, $(\Omega, C_\Omega)$ is a complete metric space if and only if $\Omega$ is dual convex. 
\end{theorem}

\begin{proof}
Suppose that $(\Omega, C_{\Omega})$ is a complete metric space. Let 
\begin{align*}
\wh{\Omega} = \{ x \in G/P : x \notin Z_{\xi} \text{ for all } \xi \in \Omega^*\}.
\end{align*}
Since $\Omega^*$ is compact,  $\wh{\Omega}$ is open and $\Omega \subset \wh{\Omega}$ by the definition of $\Omega^*$. Moreover $(\wh{\Omega})^* = \Omega^*$. 
Because $\Omega^*$ has non-empty interior, $\wh{\Omega}$ is a bounded open set in an affine chart. Then the proof of Theorem~\ref{thm:cara_metric} implies 
that $C_{\wh{\Omega}}$ is a metric on $\wh{\Omega}$ (notice that $\wh{\Omega}$ may not be connected, but connectivity is not used in the proof of Theorem~\ref{thm:cara_metric})
and 
\begin{align*}
C_{\Omega} = C_{\wh{\Omega}}|_{\Omega}.
\end{align*}

We claim that $\Omega$ coincides with a connected component of $\wh{\Omega}$ which would imply that $\Omega$ is dual convex. Suppose not, then there exists some $x \in \partial \Omega \cap \wh{\Omega}$. Pick a sequence $x_n \in \Omega$ such that $x_n \rightarrow x$. Then 
\begin{align*}
C_{\Omega}(x_n, x_{n+m}) = C_{\wh{\Omega}}(x_n, x_{n+m}) \leq C_{\wh{\Omega}}(x_n, x) +C_{\wh{\Omega}}(x, x_{n+m})
\end{align*}
so $x_n$ is a Cauchy sequence. But $C_{\Omega}$ is a complete metric and so $x \in \Omega$ which is a contradiction. Thus $\Omega$ is dual convex. 

Now suppose that $\Omega$ is dual convex. We wish to show that $C_\Omega$ is a complete metric on $\Omega$. Suppose that $x_n$ is a Cauchy sequence, then by passing to a subsequence we can suppose that 
\begin{align*}
\sum_{i=1}^\infty C_\Omega(x_i, x_{i+1}) = M < \infty.
\end{align*}
Then 
\begin{align*}
C_{\Omega}(x_n,x_1) \leq M
\end{align*}
for all $n \geq 0$. Since $G/P$ is compact we can pass to a subsequence such that $x_n \rightarrow x \in \overline{\Omega}$. We claim that $x \in \Omega$. Otherwise there exists some $\xi \in \Omega^*$ such that $x \in Z_\xi$. Since $\Omega$ is bounded in some affine chart we can find some $\eta \in \Omega^*$ such that $x \notin Z_{\eta}$. Then 
\begin{align*}
C_{\Omega}(x_1, x_n) \geq   \log \abs{ \frac{ \iota^*(\xi)\Big(\iota(x_1)\Big)\iota^*(\eta)\Big(\iota(x_n)\Big)}{\iota^*(\xi)\Big(\iota(x_n)\Big)\iota^*(\eta)\Big(\iota(x_1)\Big)}}.
\end{align*}
But since $x \in Z_\xi$ the right hand side of the above expression goes to infinity as $n \rightarrow \infty$. Thus we have a contradiction and thus $x \in \Omega$. So $C_{\Omega}$ is a complete metric on $\Omega$. 
\end{proof}

We can now prove Theorem~\ref{thm:dual_convex} from the introduction. The key step is the following variant of the Hopf-Rinow theorem:

\begin{lemma}\label{lem:hopf_rinow} Suppose $(X,d)$ is a locally compact metric space and there exists a compact set $K \subset X$ such that $X = \operatorname{Isom}(X,d) \cdot K$. Then $(X,d)$ is a complete metric space. 
\end{lemma}

\begin{proof}
We first claim that there exists $\delta >0$ such that for any $x \in X$ the set 
\begin{align*}
B_{X}(x;\delta):=\{ y \in X : d(x,y) \leq \delta\}
\end{align*}
is compact. Since $(X,d)$ is locally compact, for any $k \in K$ there exists $\delta_k >0$ such that $B_X(k;\delta_k)$ is compact. Then since 
\begin{align*}
K \subset \cup_{k \in K} \{ y \in X : d(k,y) < \delta_k/2\}
\end{align*}
there exists $k_1, \dots, k_N \in K$ such that 
\begin{align*}
K \subset \cup_{i=1}^N \{ y \in X : d(k_i,y) < \delta_{k_i}/2\}.
\end{align*}
Then if $\delta:= \min \{ \delta_{k_i}/2\}$ we see that $B_X(x;\delta)$ is compact for any $x \in X$. 

Now suppose that $x_n$ is a Cauchy sequence in $(X,d)$. Then there exists $N >0$ such that $d(x_n, x_N) < \delta$ for $n > N$. But then there exists a subsequence $x_{n_k}$ which converges. Thus $(X,d)$ is complete. 
\end{proof}

Now Theorem~\ref{thm:completeness} and the above lemma imply:

\begin{corollary}\label{cor:dual_convex}
Suppose $G$ is a connected non-compact simple Lie group with trivial center, $P \leq G$ is a parabolic subgroup, and $\Omega \subset G/P$ is a proper quasi-homogeneous domain. Then $\Omega$ is dual convex. 
\end{corollary}

\bibliographystyle{alpha}
\bibliography{hilbert}

\end{document}